\documentclass[11pt]{amsart}

\title{Computing derivations on nilpotent quadratic Lie algebras}
\author{Pilar Benito}
\email{pilar.benito@unirioja.es}
\author{Javier R\'andez-Ib\'a\~nez}
\email{jarandez@unirioja.es}
\author{Jorge Rold\'an-L\'opez}
\email{jorge.roldanl@unirioja.es}
\address{Departamento de Matem\'aticas y Computaci\'on, Universidad de La Rioja, Logro\~no, La Rioja, Spain}
\date{\today}

\makeatletter
\@namedef{subjclassname@2020}{%
  \textup{2020} Mathematics Subject Classification}
\makeatother


\usepackage[utf8]{inputenc}
\usepackage[T1]{fontenc}
\usepackage[english]{babel}

\usepackage{amsmath, amsthm, amssymb, enumerate}
\usepackage{hyperref, tikz, breakcites, todonotes}

\theoremstyle{plain}
	\newtheorem{lemma}{Lemma}[section]
	\newtheorem{proposition}[lemma]{Proposition}
	\newtheorem{theorem}[lemma]{Theorem}
	
\theoremstyle{remark}
	\newtheorem{remark}[lemma]{Remark}
	\newtheorem{example}{Example}[section]

\newcommand{\mfa}{\mathfrak{a}}
\newcommand{\mfg}{\mathfrak{g}}
\newcommand{\mfn}{\mathfrak{n}}

\newcommand{\mfFL}{\mathfrak{FL}}
\newcommand{\mfsl}{\mathfrak{sl}}
\newcommand{\mfso}{\mathfrak{so}}
\newcommand{\mfsu}{\mathfrak{su}}
\newcommand{\mbF}{\mathbb{F}}
\newcommand{\mbR}{\mathbb{R}}

\DeclareMathOperator{\ad}{ad}
\DeclareMathOperator{\id}{Id} 
\DeclareMathOperator{\spa}{span}
\DeclareMathOperator{\inner}{Inner}
\DeclareMathOperator{\rad}{Rad}
\DeclareMathOperator{\der}{Der}
\DeclareMathOperator{\soc}{soc}

\begin{document}
\begin{abstract}
	Every non-solvable and non-semisimple quadratic Lie algebra can be obtained as a double extension of a solvable quadratic Lie algebra. Thanks to a partial classification of nilpotent Lie algebras and this result, we can design different techniques to obtain any quadratic Lie algebra whose (solvable) radical ideal is nilpotent. To achieve this, we propose two alternative methods, both involving the use of quotients. In addition to their mathematical description, both approaches introduced in this paper have been computationally implemented and are publicly available to use for generating these algebras.
\end{abstract}
\maketitle

\section{Introduction}

Quadratic Lie algebras were introduced in 1957 by Tsou and Walker (see~\cite{Tsou_Walker_1957}). They are Lie algebras $L$ endowed with a nondegenerate symmetric bilinear form $\varphi$ which is invariant with respect to their Lie product, that is, for every $x,y,z \in L$,
\begin{equation*}
	\varphi([x,y],z) + \varphi(y,[x,z]) = 0.
\end{equation*}

There exists a wide variety of quadratic Lie algebras. Aside from the abelian ones, which are trivially quadratic, all semisimple Lie algebras are also quadratic. Indeed, in an algebraically closed field, the only invariant non-degenerate symmetric bilinear forms for simple Lie algebras are non-zero scalar multiples of their Killing form. Even more, the non-degeneration of the Killing form characterizes these algebras. However, studying general quadratic Lie algebras for non-reductive\footnote{A Lie algebra is reductive if its solvable radical is just its centre.} Lie algebras is a challenging task (see \cite{Ovando_2016} and references therein for a survey guide). Two different approaches have been employed over time:
\begin{itemize}
	\item Classifying a small subfamily within these quadratic Lie algebras (small dimension, metabelian, current Lie, \dots).
	\item Constructing quadratic Lie algebras through various techniques (double extension, $T^*$-extension, inflactions, two-fold extensions, quadratic families of matrices, direct sums, tensor products, \dots).
\end{itemize}
Utilizing these construction ideas, we can obtain every non-solvable and non-semisimple quadratic Lie algebra via a double extension of some solvable quadratic Lie algebra \cite[Theorem 2.2]{Bordemann_1997}. At the same time, all solvable quadratic Lie algebras can be generated by applying iteratively double extensions, starting from nilpotent quadratic Lie algebras \cite{Favre_Santharoubane_1987}. Furthermore, these nilpotent quadratic Lie algebras can be seen as quotients of some free nilpotent Lie algebras \cite[Section 4]{Benito_delaConcepcion_Laliena_2017}, opening up a path for work.

The purpose of this paper is to introduce the mathematical objects, results and methods that are useful in the theory of quadratic Lie algebras. Also it serves us to present some essential computational functions we have developed for constructing non-solvable quadratic Lie algebras with nilpotent radical. Our algorithms are in a similar vein to some of those proposed by W. A. De Graaf in \cite{DeGraaf_2000}. We are going to take as a starting point the classification of quadratic nilpotent Lie algebras with few generators and a small nilpotency index in \cite[Sections 5 and 6]{Benito_delaConcepcion_Laliena_2017}. Then, employing double extensions, we are going to obtain non-solvable quadratic algebras from these nilpotent ones. Additionally, to obtain different examples, we introduce two distinct approaches based on the general results. All the results and methods mentioned in this article are available in our GitHub repository located at
\begin{center}
	\url{https://github.com/joroldan/MathematicaLieFunctions}
\end{center}

\noindent There, you can find the Wolfram Mathematica\footnote{For occasional and sporadic usage, at this moment, you can create a Wolfram Account and execute the package online for free.} package and all its source code, which anyone can download, use, inspect or modify.

Finally, it is essential to highlight the significance and popularity of the algebras studied in this paper. For an in-depth exploration of topics in which the structure of this algebras holds special interest see \cite[Section 1]{Bordemann_1997}, \cite[Section 7]{FigueroaOFarrill_Stanciu_1996}, \cite{Baum_Kath_2003}, \cite[Sections 3 and 4]{Zusmanovich_2014} and \cite{RodriguezVallarte_Salgado_2018}. In the literature and as of the current moment, most of the examples, results and constructions refer to solvable quadratic algebras and rather few involve non-solvable quadratic ones. In addition to semisimple class, interesting examples of non-solvable quadratic algebras appear in the class of Lorentzian algebras (see~\cite[Theorem II.6.4]{Hilgert_Hofmann_Lawson_1989}) or in the class of current Lie algebras \cite{Zusmanovich_2014}. We can already find some published papers working on the examples and constructions of non-solvable and non-semisimple quadratic Lie algebras like \cite[Proposition A]{Hofmann_Keith_1986}, \cite[Section~5]{Bajo_Benayadi_1997} and \cite{Benayadi_Elduque_2014}. The techniques used in these papers are based on representation theory and tensor products of quadratic Lie algebras by commutative algebras. In contrast, as we are going to point out through the paper, our approach, which provides computational algorithms, is applicable to obtain all these algebras and, even more, it is also valid for a broader variety of algebras.

The paper splits into 7 sections. Section~\ref{s:definitions} includes the mathematical basic background that we will employ throughout the paper. In Section~\ref{s:deconstruction}, two main results on quadratic Lie algebras are presented. Those results serve as the guiding thread that justifies the relevance of quadratic nilpotent algebras on the construction of the general quadratic ones. The findings on this section lead us to Section~\ref{s:nilpotent} and \ref{s:derivations} where main computational techniques are developed and illustrated. All this results end up in the computation of skew-derivations in Section~\ref{s:computations} in order to extend nilpotent Lie algebras to produce non solvable quadratic Lie algebras. The final section serves as a conclusion.

Unless we indicate the contrary we work over a generic field $\mbF$ of characteristic zero. All vector spaces we consider are finite dimensional.

\section{Definitions and general techniques}\label{s:definitions}

First, let us define some basic terminology taking as reference the book~\cite{Jacobson_1979}. Let $\mfn$ be a Lie algebra, we can define its lower central series recursively as $\mfn^1 = \mfn$ and $\mfn^t = [\mfn^{t-1}, \mfn]$. When this series reaches zero, we say $\mfn$ is a nilpotent Lie algebra, and $t$ is its nilpotency index or nilindex whenever $\mfn^t \neq 0$ but $\mfn^{t+1} = 0$. In this case, we can also write $\mfn$ is $t$-step. Moreover, the codimension of $\mfn^2$ in $\mfn$ is referred as the type of the nilpotent Lie algebra, and it coincides with the cardinality of a minimal set of generators of the nilpotent Lie algebra. This result appeared in \cite[Section~1, Corollary~1.3]{Gauger_1973} and as an exercise in \cite[Chapter~1, Exercise~10]{Jacobson_1979} and states that $\{x_1, \dots, x_d\}$ generates $\mfn$ if and only if $\{x_1+\mfn^2, \dots, x_d+\mfn^2\}$ generates $\mfn/\mfn^2$. And, it is a minimal set of generators if at the same time $\{x_1, \dots, x_d\}$ forms a basis for some subspace $\mfg$ such that $\mfn = \mfg \oplus \mfn^2$. Moreover, in a general (non-necessarily nilpotent) Lie algebra, its largest nilpotent ideal is referred to as the nilradical of the algebra.

Next, we will denote as $\mfn_{d,t}$ the free $t$-step nilpotent Lie algebra of $d$ generators (see \cite{Benito_delaConcepcion_2013} and references therein for a formal specification and patterns). Just to give a brief definition (for the notion of free Lie algebras, we follow \cite[Chapter~V, Section~4]{Jacobson_1979}): let $\mfFL(d)$ be the free Lie algebra on a set of generators $\{x_1, \dots, x_d\}$, over a field $\mbF$. In this context, $\mfn_{d,t}$ is defined as the quotient algebra
\begin{equation*}
	\mfn_{d,t} = \mfFL(d)/\mfFL(d)^{t+1}.
\end{equation*}
The elements of $\mfFL(d)$ are linear combinations of monomials
\begin{equation*}
	[x_{i_1}, \dots, x_{i_s}] = [\dots [[x_{i_1}, x_{i_2}], x_{i_3}], \dots, x_{i_s}],
\end{equation*}
where $s \geq 1$. Hence, the free nilpotent algebra $\mfn_{d,t}$ is generated as vector space by $s$-monomials $[x_{i_1}, \dots, x_{i_s}]$, for $1 \leq s \leq t$. Although for obtaining simpler basis, in this article we are also using products of these $s$-monomials.

Similar to the lower central series, there exists another recursive series $\mfn^{(0)} = \mfn$ and $\mfn^{(t)} = [\mfn^{(t-1)}, \mfn^{(t-1)}]$. Again, when this series reaches zero, we say our algebra $\mfn$ is solvable, and by definition, the largest solvable ideal of some Lie algebra is its solvable radical. It is straightforward to verify that nilpotent implies solvable. Throughout this paper, we focus on Lie algebras whose radical coincides with their nilradical. This radical plays a crucial role in the well-known Levi decomposition, that states that any Lie algebra $L$ decomposes as $R \oplus S$ where $R$ is the solvable radical ideal and $S$ is some semisimple subalgebra.

From now on, $(\mfn, \varphi)$ will denote a quadratic nilpotent Lie algebra, and $\mfn_{d,t}$ will represent a free $t$-step nilpotent Lie algebra of $d$ generators.

Apart from the previous concepts, to refine our classification, we focus solely on indecomposable quadratic Lie algebras. A quadratic algebra $(L, \varphi)$ is called decomposable if it contains a proper ideal $I$ that is non-degenerated (i.e. $\varphi|_{I\times I}$ is non-degenerate), and indecomposable otherwise. Indecomposable algebras are part of a broader family: the reduced ones. A Lie algebra $L$ is considered reduced when $Z(L) \subseteq L^2$. Non-reduced algebras can be decomposed as expected, following \cite[Theorem 6.2]{Tsou_Walker_1957}, which states that any non-reduced and non-abelian quadratic Lie algebra $(L, \varphi)$ decomposes as an orthogonal direct sum of proper ideals for $L = L_1 \oplus \mfa$, and $\varphi = \varphi_1 \perp \varphi_2$, where $(L_1, \varphi_1)$ is a quadratic reduced Lie algebra, and $(\mfa, \varphi_2)$ is a quadratic abelian algebra. However, being reduced does not necessarily imply being indecomposable.

Besides these definitions, we need to explain the concept of double extension. This technique appeared in several independent works during the 1980s (see~\cite{Medina_Revoy_1985}, \cite{Hofmann_Keith_1986}, \cite{Favre_Santharoubane_1987}). Double extension is a multistep procedure involving two extensions: first a central extension, and later a semidirect product. To apply this method, we require a starting quadratic Lie algebra $(L,\varphi)$, another Lie algebra $B$ and a Lie homomorphism $\phi\colon B \to \der_\varphi(L)$. Here, $\der_\varphi(L)$ denotes the derivations\footnote{A derivation $d$ of a Lie algebra $L$ is an endomorphism such that for every $x,y \in L$ $d([x,y]) = [d(x),y] + [x,d(y)]$.} of $L$ such that are invariant with respect to $\varphi$, that is, $\varphi(d(x), x') + \varphi(x, d(x')) = 0$. That is, the derivation $d$ is a skew-adjoint map respect to $\varphi$. From now on, we will refer to these derivations as $\varphi$-skew. With all those three ingredients, we construct the quadratic Lie algebra $(L_B, \varphi_B) := (B \oplus L \oplus B^*, \varphi_B)$ with Lie bracket
\begin{multline*}
	[b+x+\beta,b'+x'+\beta] := [b,b'] + \phi(b)(x') - \phi(b')(x) + [x, x'] \\+ w(x,x') + \beta\circ \ad b'-\beta'\circ \ad b
\end{multline*}
and bilinear form
\begin{equation*}
	\varphi_B(b + x + \beta, b' + x' + \beta') := \beta(b') + \beta'(b) + \varphi(x,x'),
\end{equation*}
where $w\colon L \times L \to B$ maps $w(x, x')(b) := \varphi(\phi(b)(x), x')$. This procedure will be used through all the paper to obtain larger and more general algebras, and, at the same time, to deconstruct algebras applying it the other way around. In the particular case $B=\mbF \cdot b$, we have $B^*=\mbF \cdot b^*$ where $b^*(\alpha b)=\alpha$ for any $\alpha\in \mbF$, and the bracket product reduces to (here $d:=\phi(b)$)
\begin{equation*}
	[t b+x+\beta,s b+x'+\beta] := td(x') - sd(x) + [x, x']+ \varphi(d(x),x')b^*.
\end{equation*}

Finally, when dealing with quadratic Lie algebras $(L,\varphi)$, orthogonal subspaces, $U^\perp=\{x\in L: \varphi(x,U)=0\}$, become a powerful tool for studying ideals. Firstly, $I$ is an ideal if and only if $[I,I^\perp] = 0$ and $(I^\perp)^\perp = I$ (check for instance \cite{FigueroaOFarrill_Stanciu_1996}). Also, $(L^2)^\perp = Z(L)$ and vice versa (see~\cite{Tsou_Walker_1957}). In addition, given two subalgebras $L_1$ and $L_2$, we have $(L_1 \cap L_2)^\perp = L_1^\perp + L_2^\perp$ and $(L_1 + L_2)^\perp = L_1^\perp \cap L_2^\perp$. Even more, the orthogonal of the Jacobson radical $\mathcal{J}(L)$, defined originally as $[L,R] = L^2 \cap R$, coincides with $\soc(L)$, the socle of $L$, which is defined as the sum of all minimal ideals. This is a consequence of \cite{Marshall_1967} as $\mathcal{J}(L)$ can also be obtained as the intersection of all maximal ideals.

\section{Preliminary deconstruction}\label{s:deconstruction}

As mentioned in the introduction, thanks to Levi decomposition theorem, for a given Lie algebra $L$ we can decompose $L = R \oplus S$ where $R$ is the radical (largest solvable ideal of $L$) and $S$ is some semisimple subalgebra. Furthermore, when $L$ is quadratic, we can give more details in such decomposition as stated in the next proposition and theorem which are the main results of this section.

\begin{proposition}\label{prop:items}
	Let $L$ be a non-solvable quadratic Lie algebra, then
	\begin{enumerate}[\quad (i)]
		\item $R^\perp \oplus Z(L) = \soc (L)=\mathcal{J}(L)^\perp$, \label{it:pd2}
		\item $R^\perp \cong S^*$ as adjoint and coadjoint $S$-modules, \label{it:pd2m}
		\item $R^\perp \cong (L/R)^*$ as adjoint and coadjoint $L$-modules. \label{it:pd3}
	\end{enumerate}
	Even more, if $L$ has no simple ideals then:
	\begin{enumerate}[\quad (i)]
		\setcounter{enumi}{2}
		\item $Z(R)=R^\perp \oplus Z(L)$, \label{it:pd0}
		\item $L = S \oplus (S^\perp \cap R) \oplus R^\perp$. \label{it:pd1}
	\end{enumerate}
\end{proposition}
\begin{proof}
	Let us start proving the first item. Observe, $R^\perp \oplus Z(L)$ is a direct sum as if $x \in R^\perp \cap Z(L)$, then $x \in \rad(\varphi) = L^\perp = 0$ because
	\begin{equation*}
		\varphi(x, L) = \varphi(x, S + R) = \varphi(x,S) = \varphi(x,[S,S]) = \varphi([x,S],S)=0.
	\end{equation*}
	Now, as the Jacobson radical $\mathcal{J}(L) = L^2 \cap R$ is the intersection of all maximal ideals, its orthogonal $\mathcal{J}(L)^\perp = (L^2 \cap R)^\perp = (L^2)^\perp + R^\perp = Z(L) + R^\perp$ must be the sum of all minimal ideals, the socle of $L$, $\soc(L)$, as mentioned at the end of previous section. An alternative proof for this item can be found in \cite[Corollary 4.2]{Benayadi_2003}.

	For item~\eqref{it:pd2m} we use the $S$-module homomorphism $\Omega\colon R^\perp \to S^*$ defined as $x\mapsto \varphi_x$ with $\varphi_x(s)=\varphi(x,s)$ for all $s\in S$. Note that $\ker\, \Omega=R^\perp\cap S^\perp=(S\oplus R)^\perp=0$ and $\dim R^\perp=\dim L-\dim R=\dim S=\dim S^*$ . Therefore, $\Omega$ is an isomorphism.

	Next, item~\eqref{it:pd3} is a straightforward computation as $R^\perp \cong (L/R)^*$ as $L$-modules using the endomorphism $\hat{\Omega}\colon R^\perp \to (L/R)^*$ defined as $x\mapsto \hat{\varphi}_x\colon L/R\to \mathbb{F}$ with $\hat{\varphi}_x(y+R)=\varphi(x,y)$.

	The final assertions concerning no simple ideals are as follows. For item~\eqref{it:pd0}, on the one hand, by item~\eqref{it:pd2}, $R^\perp \subseteq \soc(L)$, but as there are no simple ideals, all minimal ideals are abelian, and then $\soc(L) \subseteq N(L) \subseteq R$. On the other hand, as $R$ is an ideal, $[R, R^\perp] = 0$, so $R^\perp \subseteq C_L(R)$. Thus $R^\perp\oplus Z(L) \subseteq R \cap C_L(R) = Z(R)$. For reverse content, as $Z(R)$ is an ideal, it decomposes as a direct sum of irreducible $S$-modules. And each summand is a minimal ideal as $[Z(R),R]=0$.

	The last item comes from Levi decomposition $L= S \oplus R$. There, taking orthogonal complements, as $L$ is quadratic, we obtain $L^\perp = S^\perp \cap R^\perp = 0$. Note $R^\perp \subseteq R$ because of item~\eqref{it:pd0}. Therefore, we also have $S \cap R^\perp \subseteq S \cap R = 0$, thus we can compute $(R^\perp \oplus S) \cap (R^\perp \oplus S)^\perp = (R^\perp \oplus S) \cap R \cap S^\perp = 0$. This leads to $L = (R^\perp \oplus S) \oplus (R^\perp \oplus S)^\perp = R^\perp \oplus S \oplus (R \cap S^\perp)$. This leads to $L = (R^\perp \oplus S) \oplus (R^\perp \oplus S)^\perp = R^\perp \oplus S \oplus (R \cap S^\perp)$.
\end{proof}

For any $(L, \varphi)$ quadratic Lie algebra with $R^\perp \subseteq R$ and Levi factor $S$, we introduce the Lie homomorphism:
\begin{alignat*}{2}
	\Phi_S \colon S \to &  & \der_{\hat{\varphi}}(R/R^\perp) & \nonumber                        \\
	x \mapsto           &  & \ \Phi_S(x) \colon R/R^\perp    & \to R/R^\perp                    \\\
	                    &  & y+R^\perp                       & \mapsto [x, y]+R^\perp.\nonumber
\end{alignat*}
Next theorem completes \cite[Theorem~2.2]{Bordemann_1997} by giving the exact isomorphism and bilinear form for achieving isometry.

\begin{theorem}\label{thm:doubleDeconstruction}
	Let $(L, \varphi)$ be a quadratic Lie algebra with no simple ideals and $\hat{\varphi}(x+R^\perp,y+R^\perp)=\varphi(x,y)$. Then $R^\perp \subseteq R$ and $L$ is isomorphic to $((R/R^\perp)_S = S \oplus R/R^\perp \oplus S^*,\hat{\varphi}_S)$ the double extension of $(R/R^\perp, \hat{\varphi})$ by $(S, \Phi_S)$. And it is isometrically isomorphic when considering the bilinear form $\hat{\varphi}_S + \phi$, where $\phi$ is a bilinear symmetric form such that $\phi|_{S\times S} = \varphi|_{S\times S}$ and $\phi|_{(R/R^\perp \oplus S^*) \times (R/R^\perp)_S} = 0$.
\end{theorem}
\begin{proof}
	From Proposition \ref{prop:items}, we have $R^\perp \subseteq R$. Denote $E$ as the Lie algebra obtained in that double extension, i.e., $E = S \oplus R/R^\perp \oplus S^*$. Now, by item~\eqref{it:pd1} in Proposition~\ref{prop:items} we have $L = S \oplus (S^\perp \cap R) \oplus R^\perp$.
	This decomposition of $L$ allows us to define $\Psi \colon L \to E$ as follows:
	\begin{equation*}
		\Psi|_S = \id_S, \Psi|_S = \id_S \text{\ and\ }\Psi|_{R^\perp} = \Omega|_{R^\perp},
	\end{equation*}
	where $\Omega$ is a $S$-module homomorphism $\Omega\colon R \to S^*$ (adjoint and coadjoint modules) defined as $x\mapsto \varphi_x$ with $\varphi_x(s)=\varphi(x,s)$ for all $s\in S$. Note that $\Omega|_{R^\perp}$ is analogous to function $\Omega$ introduced in the proof of Proposition~\ref{prop:items} and $\ker\, \Omega = S^\perp \cap R$. This $\Psi$ is a Lie algebra isomorphism as $\dim E=\dim L$ and $\ker \Psi = \{x \in R \cap S^\perp : x \in R^\perp \} = 0$.

	Finally, we can observe that if we take $f = \hat{\varphi}_S + \phi$ as the bilinear form of the double extension $E$, then $f(\Psi(x), \Psi(y)) = \varphi(x,y)$ for every $x,y \in L$, thus it is an isometry.
\end{proof}

Thanks to the previous result and \cite[Theorem~6.2]{Tsou_Walker_1957} we can reduce the construction of generic quadratic Lie algebras to just the semisimple, abelian and solvable-reduced ones separately
\begin{equation*}
	\underbrace{(L,\varphi)}_{\text{Mixed}}=\underbrace{(L_0,\varphi_0)}_{\text{Semisimple ideal}}\perp \underbrace{(\mfa,\varphi_1)}_{\text{Abelian ideal}}\perp \underbrace{(L_1,\varphi_1)}_{\text{Double extension of a solvable}}
\end{equation*}

The semisimple ideal part $L_0$, also named as semisimple socle, is located inside $R^\perp$. The abelian ideal part is just the central ideal outside of $L^2$. Removing these two ideals, we arrive at the reduced quadratic algebra $(L_1, \varphi_1)$ which, if not solvable, is a double extension of the solvable quadratic $(\frac{R_1}{R_1^\perp}, \hat{\varphi_1})$ by a Levi a subalgebra $S_1$ of $L_1$ applying Theorem \ref{thm:doubleDeconstruction}. Figure \ref{fig:chainIdeals} displays a sequence of ideals for the algebra $(L_1, \varphi_1)$ from our preliminary deconstruction (from non-solvable to solvable) that play a significant role in its reconstruction.

Assume next $(L_1, \varphi_1)$ is a solvable-reduced quadratic Lie algebra, so $L_1=R_1$, $R_1^\perp=0$, $\mathcal{J}_1=L_1^2$ and $\mathcal{J}_1^\perp=Z(L_1)\subseteq L_1^2$. We have two types of constructions for $L_1$:
\begin{itemize}
	\item According to \cite[Proposition 2.9]{Favre_Santharoubane_1987}, $L_1$ contains a totally isotropic central element $z\in Z(L_1)$ that allows reconstructing $L_1$ as a one-dimensional double extension of the solvable algebra $\frac{I^\perp}{I}$ where $I=\spa \langle z\rangle$ is a minimal ideal and then $I^\perp$ is maximal of codimension $1$. So, solvable quadratic algebras can be built throughout a chain of double extensions that involved totally isotropic central elements. This elements are located in nilradicals.
	\item Following \cite[Proposition 5.61]{Keith_1984}, $L_1$ is a central bi-extension of the nilpotent algebra $(\frac{N_1}{N_1^\perp}, \hat{\varphi_1})$, where $N_1$ is the nilradical of $L_1$ and $\hat{\varphi_1}$ is the non-degenerate form induced by $\varphi_1$.
\end{itemize}
\begin{figure}
	\centering
	\begin{tikzpicture}[node/.style={circle, fill=black, inner sep=1.5pt, outer sep=2pt}]
		\node[node, label=above:$L_1$]   (L)  at (0.0,0) {};
		\node[node, label=above:$R_1$]   (R)  at (1.5,0) {};
		\node[node, label=above:$N_1$]   (N)  at (3.0,0) {};
		\node[node, label=above:$\mathcal{J}_1$] (R2) at (4.5,0) {};
		\node[node, label=above:$\mathcal{J}_1^\perp$] (R2P) at (6.0,0) {};
		\node[node, label=above:$N_1^\perp$] (NP) at ( 7.5,0) {};
		\node[node, label=above:$R_1^\perp$] (RP) at ( 9.0,0) {};
		\node[node, label=above:$0$]         (O)  at (10.5,0) {};

		\draw (O) -- (RP) -- (NP) -- (R2P) -- (R2) -- (N) -- (R) -- (L);
	\end{tikzpicture}
	\caption{Here $N_1$ is the nilradical of $L_1$, $\mathcal{J}_1=\mathcal{J}(L_1)$ and $\mathcal{J}_1^\perp=R_1^\perp \oplus Z(L)$. For $L_1=R_1$, we have $R_1^\perp=0$.}
	\label{fig:chainIdeals}
\end{figure}
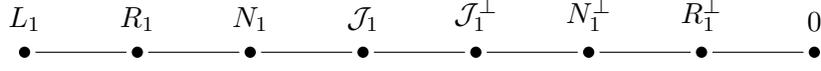

We point out that the discussed deconstructions lead as up to the nilradical. As we will focus on the construction of nonsolvable quadratic Lie algebras, this suggests a procedure for constructing non-solvable algebras starting from quadratic nilpotents.
\begin{enumerate}
	\item Let $(\mfn, \varphi)$ a nilpotent quadratic Lie algebra. (All of them appears by using free nilpotent Lie algebras and their invariant bilinear forms \cite{Benito_2017}.)
	\item Compute the algebra of derivations, $\der \mfn$, and the subalgebra $\der_\varphi \mfn$ of $\varphi$-skew derivations.
	\item Check the existence of subalgebras $S$ of $\der_\varphi \mfn$ different from  those inside the ideal $\inner \mfn$. Look for the semisimple ones. (According to \cite[Proposition 5.1]{FigueroaOFarrill_Stanciu_1996}, one dimensional doble extension through an inner derivation does not produce new algebras. Semisimple subalgebras exist within the Levi subalgebras (see \cite{DeGraaf_2000}).)
	\item Apply the double extension procedure given in Section~\ref{s:definitions} starting from the natural inclusion $\iota\colon S\hookrightarrow \der_\varphi \mfn$, so $\Phi_S=\iota$.
\end{enumerate}

Our following examples highlights all those algorithmic ideas.
\begin{example}
	In case $R=R^\perp$, as $R$ is an ideal, $[R, R^\perp]=0$. It follows that, $R^2 = 0$, making $R$ is abelian. Consequently, a Levi decomposition takes the form $L = S \oplus R = S \oplus R^\perp= S \oplus S^*$. This type of algebras are studied as particular cases of either inflaction constructions proposed in~\cite{Keith_1984} or as $T^*$-extension constructions given in~\cite{Bordemann_1997} and \cite[Corollary~2.5]{Benayadi_Elduque_2014}.
\end{example}

\begin{example}
	For $n\geq 2$, consider the Euclidean even dimensional real vector space $(V, \varphi)$ and $e_1,\dots, e_{2n}$ any orthonormal basis. It is easily checked that the linear map $d(e_{2i-1})=-e_{2i}$ and $d(e_{2i})=-e_{2i-1}$ is a skew-adjoint automorphism. Considering $V$ as an abelian Lie algebra, $(V, \varphi)$ is a quadratic and $d\in \der_\varphi (V)$. So we can consider the double extended quadratic algebra $A_{2n+2}=\mbF\cdot d \oplus (V, \varphi) \oplus \mbF\cdot d^*$ with invariant bilinear form. According to \cite[Theorem 3.2]{Benito_RoldanLopez_2023} the Levi factor of the algebra $\der_\varphi (A_{2n+2})$ is the simple real algebra $\mfsu_n(\mbR)$ and then we get the series of non-solvable and non-semisimple Lie algebras:
	\begin{equation*}
		L_1(n)=\mfsu_n(\mbR) \oplus A_{2n+2} \oplus \mfsu_n(\mbR)^*
	\end{equation*}
	It is easily checked that $R(L_1(n))=A_{2n+2} \oplus \mfsu_n(\mbR)^*$, $R(L_1(n))^\perp= \mfsu_n(\mbR)^*$, $N(L_1(n))=(V, \varphi) \oplus \mbF\cdot d^* \oplus \mfsu_n(\mbR)^*$ and $N(L_1(n))^\perp=\mbF\cdot d^* \oplus \mfsu_n(\mbR)^*$. Moreover, $N(L_1(n))=\mathcal{J}(L_1(n))$. In Figure \ref{fig:deconstruction} we locate all these ideals.
	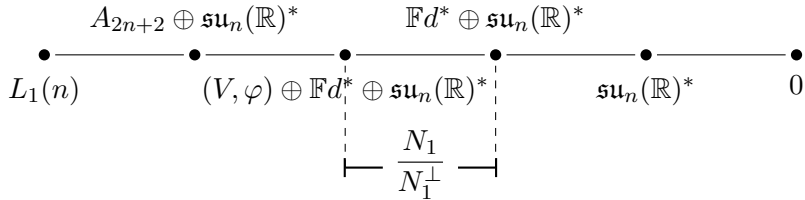
\begin{figure}[h]
	\centering
	\begin{tikzpicture}[node/.style={circle, fill=black, inner sep=1.5pt, outer sep=2pt}]
		\node[node, label=below:$L_1(n)$] (L)  at (0,0) {};
		\node[node, label=above:$A_{2n+2} \oplus \mfsu_n(\mbR)^*$]    (V)  at (2,0) {};
		\node[node, label=below:{$(V,\varphi) \oplus \mbF d^* \oplus \mfsu_n(\mbR)^*$}]    (A)  at (4,0) {};
		\node[node, label=above:$\mbF d^* \oplus \mfsu_n(\mbR)^*$] (FSU) at (6,0) {};
		\node[node, label=below:$\mfsu_n(\mbR)^*$] (SU) at (8,0) {};
		\node[node, label=below:$0$]         (O)  at (10,0) {};

		\draw (O) -- (SU) -- (FSU) -- (A) -- (V) -- (L);
		
		\coordinate (B1S) at (3.98,-1.5) {};
		\coordinate (B1S2) at (4,-1.5) {};
		\coordinate (B1E) at (4.5,-1.5) {};
		\coordinate (B2S) at (5.5,-1.5) {};
		\coordinate (B2E) at (6.02,-1.5) {};
		\coordinate (B2E2) at (6,-1.5) {};
		\node (N) at (5,-1.46) {$\dfrac{N_1}{N_1^\perp}$};
		
		\draw[dashed] (B1S2) -- (A);
		\draw[dashed] (B2E2) -- (FSU);
		\draw[thick, |-] (B1S) -- (B1E);
		\draw[thick, -|] (B2S) -- (B2E);
	\end{tikzpicture}
	\caption{Deconstruction of $L_1(n)$.}
	\label{fig:deconstruction}
\end{figure}
\end{example}

\section{Nilpotent quadratic Lie algebras}\label{s:nilpotent}

To classify nilpotent Lie algebras, we can start from free nilpotent Lie algebras and make quotients of them. According to \cite[Section~1, Propositions~1.4 and 1.5]{Gauger_1973}, any $t$-step nilpotent Lie algebra of type $d$ is isomorphic to $\mfn_{d,t}/I$ for some ideal $I$ such that $\mfn_{d,t}^t \not\subseteq I \subsetneq \mfn_{d,t}^2$. Note this last condition appears in order not to reduce the type or nilpotency index through the quotient.

\subsection{Hall bases}

One classical way of working with bases in free nilpotent Lie algebras is using the so called Hall bases. These bases were introduced in~\cite{Hall_1950}. From the generator set $\{x_1, \dots, x_d\}$, we easily get the standard $s$-monomials $[x_{i_1}, \dots, x_{i_s}]$ and combinations of them that linearly generate the Lie algebra $\mfn_{d,t}$. However, the anticommutativity law ($[x_i, x_j] + [x_j, x_i] = 0$) and the Jacobi identity both establish linear dependency relationships. Starting with a total order $x_d < x_{d-1} < \dots < x_1$, the definition of a Hall basis checks recursively if a given standard monomial depends on the previous ones. To do so, it imposes a monomial order based on the total order mentioned earlier, respecting the degree, where smaller degree means a smaller monomial. Now, a monomial $[u, v]$ belongs to Hall basis if both the following conditions apply:
\begin{itemize}
	\item both $u$ and $v$ are terms of the Hall basis and $u > v$,
	\item if $u = [z,w]$ then $v \geq w$.
\end{itemize}
This can be computationally implemented as seen in the algorithm from Table~\ref{tab:HallBasisAlg}. This algorithm is derived directly from the definition and it is similar to the one found in \cite{DeGraaf_2000}. In particular, we have implemented it for Wolfram Mathematica, and it is available in our GitHub repository mentioned in the introduction. The basis can be obtained using the following functions:
\begin{itemize}
	\item \texttt{HallBasisLevel[d,t]} returns the degree $t$ monomials in the Hall basis with $d$ generators.
	\item \texttt{HallBasisUntilLevel[d,t]} returns the full $\mathcal{H}_{d,t}$, the notation we use to represent the Hall basis of $\mfn_{d,t}$.
\end{itemize}
Both methods have a sibling function that obtains the dimension of each of them, just by adding the word \texttt{Dimension} at the end of their names. In Table~\ref{tab:HallBasisAlg2}, we provide some examples of Hall bases $\mathcal{H}_{d,t}$, for some small $d$ and $t$ values. Additionally, despite their notation describing exactly how the Lie product works, we can obtain the adjoint matrices of that product in the Hall basis computationally with the method \texttt{HallBasisAdjointList[d,t]}.

\begin{table}[h]
	\begin{tabular}{p{0.968\linewidth}}
		\noalign{\smallskip}\hline\noalign{\smallskip}
		{\small
			\textit{isCanonical}($v$):

			\setlength{\parindent}{8pt}
			\textbf{if} $\deg v == 1$ \textbf{ return true};

			\textbf{else if} \textbf{not}(\textit{isCanonical}($v_1$)\textbf{ and }\textit{isCanonical}($v_2$)\textbf{ and }$v_2 \leq v_1$)\textbf{ return false};

			\textbf{else if} $\deg v_1 > 1$\textbf{ return }(\textit{isCanonical}($v_{1,1}$)\textbf{ or }\textit{isCanonical}($v_{1,2}$)\textbf{ or }$v_2 \geq v_{1,2}$);

			\textbf{else return true};
		} \\
		\noalign{\smallskip}\hline\noalign{\smallskip}
	\end{tabular}
	\caption{Hall basis algorithm pseudocode. Note that this is a recursive algorithm. Here $v = [v_1, v_2]$ and $v_1 = [v_{1,1}, v_{1,2}]$. In order to generate Hall Basis elements of degree $n$ we can combine $v_1$ and $v_2$ in level $n-k$ and $k$ respectively, where $k = 1, \ldots, n/2$, and later removed those non-canonical.}
	\label{tab:HallBasisAlg}
\end{table}

\begin{table}[h!]
	{\small
		\tabcolsep=4pt
		\centering
		\begin{tabular}{p{0.77cm}p{10.85cm}}
			\hline\noalign{\smallskip}
			$(d,t)$  & $\mathcal{H}_{d,t}$                                                                                           \\\noalign{\smallskip}\hline\noalign{\smallskip}
			$(2, 6)$ & $x_2, x_1, [x_1, x_2], [[x_1, x_2], x_2], [[x_1, x_2], x_1], [[[x_1, x_2], x_2], x_2],$                       \\
			         & $ [[[x_1, x_2], x_2], x_1], [[[x_1, x_2], x_1], x_1], [[[[x_1, x_2], x_2], x_2], x_2],$                       \\
			         & $ [[[[x_1, x_2], x_2], x_2], x_1], [[[x_1, x_2], x_2], [x_1, x_2]],[[[[x_1, x_2], x_2], x_1], x_1], $         \\
			         & $ [[[x_1, x_2], x_1], [x_1, x_2]], [[[[x_1, x_2], x_1], x_1], x_1],[[[[[x_1, x_2], x_2], x_2], x_2], x_2], $  \\
			         & $ [[[[[x_1, x_2], x_2], x_2], x_2], x_1], [[[[x_1, x_2], x_2], x_2], [x_1, x_2]],$                            \\
			         & $  [[[[[x_1, x_2], x_2], x_2], x_1], x_1],[[[[x_1, x_2], x_2], x_1], [x_1, x_2]], $                           \\
			         & $ [[[[[x_1, x_2], x_2], x_1], x_1], x_1], [[[x_1, x_2], x_1], [[x_1, x_2], x_2]],  $                          \\
			         & $ [[[[x_1, x_2], x_1], x_1], [x_1, x_2]], [[[[[x_1, x_2], x_1], x_1], x_1], x_1] $                            \\
			\noalign{\smallskip}
			$(4, 3)$ & $x_4, x_3, x_2, x_1, [x_3, x_4], [x_2, x_4], [x_2, x_3], [x_1, x_4], [x_1, x_3], [x_1, x_2],$                 \\
			         & $[[x_3, x_4], x_4], [[x_3, x_4], x_3], [[x_3, x_4], x_2], [[x_3, x_4], x_1], [[x_2, x_4], x_4],$              \\
			         & $ [[x_2, x_4], x_3], [[x_2, x_4], x_2], [[x_2, x_4], x_1], [[x_2, x_3], x_3], [[x_2, x_3], x_2], $            \\
			         & $[[x_2, x_3], x_1], [[x_1, x_4], x_4], [[x_1, x_4], x_3],[[x_1, x_4], x_2], [[x_1, x_4], x_1],$               \\
			         & $[[x_1, x_3], x_3], [[x_1, x_3], x_2], [[x_1, x_3], x_1], [[x_1, x_2], x_2], [[x_1, x_2], x_1] $              \\
			\noalign{\smallskip}
			$(6, 2)$ & $x_6, x_5, x_4, x_3, x_2, x_1, [x_5, x_6], [x_4, x_6], [x_4, x_5], [x_3, x_6], [x_3, x_5], [x_3, x_4],$       \\
			         & $[x_2, x_6], [x_2, x_5], [x_2, x_4], [x_2, x_3], [x_1, x_6], [x_1, x_5], [x_1, x_4], [x_1, x_3], [x_1, x_2] $ \\
			\noalign{\smallskip}
			\noalign{\smallskip}\hline\noalign{\smallskip}
		\end{tabular}}
	\caption{Hall basis of $\mfn_{d,t}$. Note from expanded basis $\mathcal{H}_{4,3}$ and $\mathcal{H}_{2,6}$ we can recover Hall basis of $\mfn_{4,2}$ and $\mfn_{2,t}$ for $t=2,3,4,5$}
	\label{tab:HallBasisAlg2}
\end{table}

\subsection{Classification}

In the case of quadratic nilpotent Lie algebras, the non-degeneration of the bilinear form gives us exactly which is the ideal we need to use for the quotient. This result appears in \cite[Proposition~4.1]{Benito_delaConcepcion_Laliena_2017}. The approach consists of finding bilinear invariant (degenerate) forms for free nilpotent Lie algebras, and making the quotient by their respective radicals gives us the desired quadratic nilpotent Lie algebras. In fact, employing automorphisms, we can also determine when two quotients are isometrically isomorphic. Using these quotients, the authors were able to obtain a classification of quadratic nilpotent Lie algebras on $2,3$ generators as seen in the next theorem.

\begin{theorem}\label{thm:clas_comp}
	Up to isomorphism, the indecomposable quadratic nilpotent Lie algebras over any algebraically closed field $\mbF$ of characteristic zero and of type 1 or 2 and nilindex less or equal than 5, or of type 3 and nilindex less or equal than 3 are isomorphic to one of the following Lie algebras:
	\begin{enumerate}[\quad (a)]
		\item 1-dim. abelian Lie algebra $(\mfn_{1}, \varphi)$ with basis $\{a_1\}$ and $\varphi(a_1,a_1)=1$,
		\item 5-dimensional free nilpotent Lie algebra $(\mfn_{2}, \varphi)$ with basis $\{a_i\}_{i=1}^5$ where, for $i\leq j$, we find the non-zero products
		      \begin{align*}
			      [a_2,a_1] = a_3,\qquad
			      [a_3,a_1] = a_4,\qquad
			      [a_3,a_2] = a_5
		      \end{align*}
		      and $\varphi$ is defined as $\varphi(a_i,a_j) = (-1)^{i-1}$ when $i+j=6$ and zero otherwise.

		\item 7-dimensional Lie algebra $(\mfn_3,\varphi)$ with basis $\{a_i\}_{i=1}^7$ where, for $i\leq j$, we find the non-zero products
		      \begin{alignat*}{3}
			      [a_2, a_1] & = a_3,\qquad
			      [a_3, a_1] & = a_4,\qquad
			      [a_4, a_1] & = a_5,\\
			      [a_5, a_1] & = a_6,\qquad
			      [a_5, a_2] & = a_7,\qquad
			      [a_3, a_4] & = a_7,
		      \end{alignat*}
		      and $\varphi$ is defined as $\varphi(a_i,a_j) = (-1)^i$ when $i+j=8$ and zero otherwise.

		\item 8-dimensional Lie algebra $(\mfn_4,\varphi)$ with basis $\{a_i\}_{i=1}^8$ where, for $i\leq j$, we find the non-zero products
		      \begin{alignat*}{3}
			      [a_2 , a_1] & = a_3,\qquad [a_3 , a_1] & = a_4,\qquad [a_3 , a_2] & = a_5, \\
			      [a_4 , a_1] & = a_6,\qquad [a_6 , a_1] & = a_7,\qquad [a_6 , a_2] & = a_8, \\
			      [a_5 , a_2] & = a_6,\qquad [a_3 , a_4] & = a_8,\qquad [a_5 , a_3] & = a_7,
		      \end{alignat*}
		      and $\varphi(a_i , a_j) = (-1)^i$ for $0 \leq i \leq 3$ and $i + j = 9$, $\varphi(a_4, a_4) = \varphi(a_5, a_5) =1$ and $\varphi(a_i, a_j) = 0$ otherwise.

		\item 6-dimensional Lie algebra $(\mfn_5,\varphi)$ with basis $\{a_i\}_{i=1}^6$ where, for $i\leq j$, we find the non-zero products
		      \begin{equation*}
			      [a_2 , a_1] = a_4,\qquad [a_3 , a_1] = a_5,\qquad [a_3 , a_2] = a_6,
		      \end{equation*}
		      and $\varphi$ is defined as $\varphi(a_i,a_j) = (-1)^{i-1}$ when $i+j=7$ and zero otherwise.

		\item 8-dimensional Lie algebra $(\mfn_6,\varphi)$ with basis $\{a_i\}_{i=1}^8$ where, for $i\leq j$, we find the non-zero products
		      \begin{alignat*}{3}
			      [a_2 , a_1] & = a_4,\qquad [a_3 , a_1] & = a_5,\qquad [a_4 , a_1] & = a_6, \\
			      [a_4 , a_2] & = a_7,\qquad [a_5 , a_1] & = a_8,\qquad [a_5 , a_3] & = a_7,
		      \end{alignat*}
		      and $\varphi(a_4 , a_4) = \varphi(a_5, a_5) = \varphi(a_1 , a_7) = -\varphi(a_2, a_6) =-\varphi(a_3 , a_8) = 1$, and $\varphi(a_i, a_j) =0$ otherwise.

		\item 9-dimensional Lie algebra $(\mfn_7,\varphi)$ with basis $\{a_i\}_{i=1}^9$ where, for $i\leq j$, we find the non-zero products
		      \begin{alignat*}{3}
			      [a_2 , a_1] & = a_4,\qquad [a_3 , a_1] & = a_5,\qquad [a_3 , a_2] & = a_6, \\
			      [a_4 , a_1] & = a_7,\qquad [a_4 , a_2] & = a_8,\qquad [a_5 , a_1] & = a_9, \\
			      [a_5 , a_3] & = a_8,\qquad [a_3 , a_6] & = a_7,\qquad [a_6 , a_2] & = a_9,
		      \end{alignat*}
		      and $\varphi(a_4 , a_4 ) = \varphi(a_5 , a_5) = \varphi(a_6 , a_6) = \varphi(a_1 , a_8) = - \varphi(a_2 , a_7) = - \varphi(a_3 , a_9) = 1$, and $\varphi(a_i, a_j) =0$ otherwise.

	\end{enumerate}
	Any non-abelian quadratic Lie algebra of type less or equal than 2 is indecomposable. The non-abelian decomposable Lie algebras of type 3 and nilpotent index less or equal than 3 are the orthogonal sum $(n_{1,1},\varphi) \oplus L$ where $L$ is a quadratic Lie algebra as in item (b), (c), or (d).
\end{theorem}

In \cite[Theorem~6.1]{Benito_delaConcepcion_Laliena_2017}, there is also an analogous result over the real field. In that case, the number of algebras up to isometric isomorphisms, with the same restrictions, grows up to 16, from the 7 algebras found in the complex field.

\subsection{Invariant bilinear forms}\label{s:invariant}

The idea behind the previous classification, as stated in the original publication, is to find the radical of bilinear symmetric invariant forms, which gives the ideal for the desired quotient. Again, this can be computationally tackled. Once we have a defined basis and understand how the Lie product works in that basis, for instance corresponding adjoint matrices, the restrictions imposed on some bilinear form to be symmetric and invariant are derived from linear equations. Therefore, for some generic matrix $A$ in that same basis $\{v_1, \dots, v_n\}$, the matrix must satisfy:
\begin{itemize}
	\item $A = A^t$, as the bilinear form $\varphi$ associated must be symmetric.
	\item $B^tA = -AB$ for any matrix $B$ representing some $\ad v_i$. This happens as the invariant rule $\varphi([x,y],z) = -\varphi([y,[x,z])$ is equivalent to $\varphi(\ad_x(y), x) = -\varphi(y, \ad_x(z))$.
\end{itemize}
We can computationally obtain the adjoint matrices associated with the Lie bracket of any free nilpotent Lie algebra $\mfn_{d,t}$, and subsequently, obtain all their invariant symmetric bilinear forms using our method \texttt{GetSymmetricInvariantBilinearForm[var,adjointList]}, in combination with the result from \texttt{HallBasisAdjointList[d,t]}. Note that this method is much more general and can be used over any algebra, not only the free nilpotent ones.

However, in case we know the final algebra $\mfn$ after the quotient, we can still obtain the isomorphic $\mfn_{d,t}/I$ using a simple technique. The first step consists on finding both the type $d = \dim(\mfn/\mfn^2)$ and $t$, which is the nilpotency index, from the algebra $\mfn$. Later, the Hall basis $\mathcal{H}_{d,t}$ lists exactly which products should give us linearly independent terms. Precisely, those non-independent terms form the basis of the ideal $I$. This is also implemented in our library under the name \texttt{GetNilpotentIdeal[adjointList]}, which receives the list of adjoint matrices of some basis of algebra $\mfn$, and returns the whole $\mfn_{d,t}/I$, i.e., $d$, $t$, and $I$.

\section{Derivations}\label{s:derivations}

As we want to obtain non-solvable quadratic Lie algebras, we are going to apply double extensions to those nilpotent Lie algebras. To get the extensions, we first need to obtain their $\varphi$-skew derivations, which are essential ingredients for producing double extensions. Among them, we need to find the ones which are not inner\footnote{An inner derivation is an adjoint map $\ad x$, which is defined as $(\ad x) (y) = [x,y]$.} derivations, as they would produce nilpotent decomposable quadratic Lie algebras according to \cite[Proposition 5.1]{FigueroaOFarrill_Stanciu_1996}. There are two main approaches to tackle this problem: quotient of derivations or direct computation.

\subsection{Derivations in the quotient}
All algebras in the classification from Theorem~\ref{thm:clas_comp} can be obtained, as mentioned previously as quotients of free nilpotent Lie algebras. As explained there, those ideals for the quotient can be easily obtained, even computationally. Taking this into account, algebras in the classification can be obtained as
\begin{gather*}
	\mfn_1 \cong \mfn_{1,1}\qquad
	\mfn_2 \cong \mfn_{2,3}\qquad
	\mfn_3 \cong \mfn_{2,5}/I_1\qquad
	\mfn_4 \cong \mfn_{2,5}/I_2\\[0.1cm]
	\mfn_5 \cong \mfn_{3,2}\qquad
	\mfn_6 \cong \mfn_{3,3}/I_3\qquad
	\mfn_7 \cong \mfn_{3,3}/I_4
\end{gather*}
where
\begin{align*}
	I_1 & = \spa\langle
	[[x_1, x_2], x_1], 
	[[x_1, x_2], x_2], x_1], 
	[[x_1, x_2], x_1], x_1],\\
	    & \enspace\quad[[[x_1, x_2], x_2], x_2], x_1] + [[x_1, x_2], x_2], [x_1, x_2]], \\ 
	    & \enspace\quad[[[x_1, x_2], x_2], x_1], x_1],
	[[x_1, x_2], x_1], [x_1, x_2]],
	[[[x_1, x_2], x_1], x_1], x_1]
	\rangle,\\[0.2cm]
	I_2 & = \spa\langle
	[[[x_1, x_2], x_2], x_1], [[[[x_1, x_2], x_2], x_1], x_1],\\&\enspace\quad
	[[[x_1, x_2], x_1], x_1] - [[[x_1, x_2], x_2], x_2],\\&\enspace\quad
	[[[[x_1, x_2], x_2], x_2], x_1] + [[[x_1, x_2], x_2], [x_1, x_2]],\\&\enspace\quad
	[[[x_1, x_2], x_1], [x_1, x_2]] - [[[[x_1, x_2], x_2], x_2], x_2],\\&\enspace\quad
	[[[[x_1, x_2], x_1], x_1], x_1] - [[[[x_1, x_2], x_2], x_2], x_1]\rangle,\\[0.2cm]
	I_3 & =\spa\langle
	[x_1,x_2],
	[[x_2,x_3],x_1],
	[[x_1,x_3],x_2],
	[[x_1,x_2],x_2],\\&\enspace\quad
	[[x_1,x_2],x_1],
	[[x_2,x_3],x_2] - [[x_1,x_3],x_1]\rangle,\\[0.2cm]
	I_4 & = \spa\langle
	[[x_2,x_3],x_1],
	[[x_1,x_3],x_2],
	[[x_2,x_3],x_2] - [[x_1,x_3],x_1],\\&\enspace\quad
	[[x_1,x_3],x_3] - [[x_1,x_2],x_2],
	[[x_2,x_3],x_3] + [[x_1,x_2],x_1]\rangle.
\end{align*}
\begin{remark}
	Beyond our restrictions to the type and nilpotency index, the only free nilpotent quadratic Lie algebras are $\mfn_{d,1}$ (the abelian family), $\mfn_{2,3}$ and $\mfn_{3,2}$.
\end{remark}
This relationship with the free nilpotent Lie algebras is a good starting point for finding their derivations, or $\varphi$-skew derivations. This is a consequence of the following proposition from \cite[Proposition~5]{Sato_1971}. For any ideal $I$ of $\mfn_{d,t}$ such that $\mfn_{d,t}^t \not \subseteq I \subsetneq \mfn_{d,t}^2$, let $\der_{I} {\mfn_{d,t}}$ and $\der_{\mfn_{d,t}, I} \mfn_{d,t}$ denote the subset of derivations which map $I$ into itself and $\mfn_{d,t}$ into $I$, respectively. Both sets are subalgebras of $\der \mfn_{d,t}$, even furthermore, $\der_{\mfn_{d,t},I} \mfn_{d,t}$ is an ideal inside $\der_{I} \mfn_{d,t}$. The following result holds:

\begin{theorem}\label{thm:quotientDerivations}
	Let $I$ be an ideal of $\mfn_{d,t}$ such that $\mfn_{d,t}^t \not \subseteq I \subsetneq \mfn_{d,t}^2$. The algebra of derivations of $\dfrac{\mfn_{d,t}}{I}$ is isomorphic to $\dfrac{\der_{I} \mfn_{d,t}}{\der_{\mfn_{d,t}, I} \mfn_{d,t}}$, where $\der_{I} \mfn_{d,t}$ and $\der_{\mfn_{d,t}, I} {\mfn_{d,t}}$ map $I$ and $\mfn_{d,t}$ into $I$, respectively.
\end{theorem}
As all elements in the Hall basis of a free nilpotent Lie algebra are formed by products of the generators, its derivations can be obtained by applying the derivation rule, i.e., $D([x,y]) = [D(x),y] + [x,D(y)]$. These derivations can be easily obtained with the help of a computer for our family of algebras. Thanks to the previous theorem, we can compute what the derivations in the quotient are. A useful computational approach for finding those derivations of $\mfn_{d,t}/I$ is explained in the next steps:

\begin{enumerate}
	\item Find generic derivations $\delta_1$ and $\delta_2$ in $\mfn_{d,t}$.
	\item Change the basis in that derivation to a basis formed by the union of a basis of a complement of $I$ and a basis of $I$. This splits the derivation into 4 submatrices considering projections.
	\item Impose $\delta_1(I) \subseteq I$ and $\delta_2(\mfn_{d,t}) \subseteq I$, which refers to the upper right and upper half submatrices, respectively.
	\item After seeing both derivations, we apply the quotient (equivalence class).
\end{enumerate}

In order to compute all those derivations, we have developed the following methods in our library:
\begin{itemize}
	\item \texttt{GetDerivation[var, adjointList]}
	\item \texttt{GetInnerDerivation[var, adjointList]}
	\item \texttt{GetSkewDerivation[var,adjointList,B]}
\end{itemize}
All methods receive a variable to use, and a list of adjoint matrices encoding the structure constants of the desired algebra. The last method of all also asks for the matrix of the bilinear form in order to compute the $\varphi$-skew derivations.

\begin{example}\label{ex:tec1}
	In Figure~\ref{fig:mat1}, we can see a Mathematica code in action to find $\varphi$-skew derivations of $\mfn_{2,5}/I_1$ using Theorem~\ref{thm:quotientDerivations}.
	\begin{figure}[h!]
		\centering
		\includegraphics{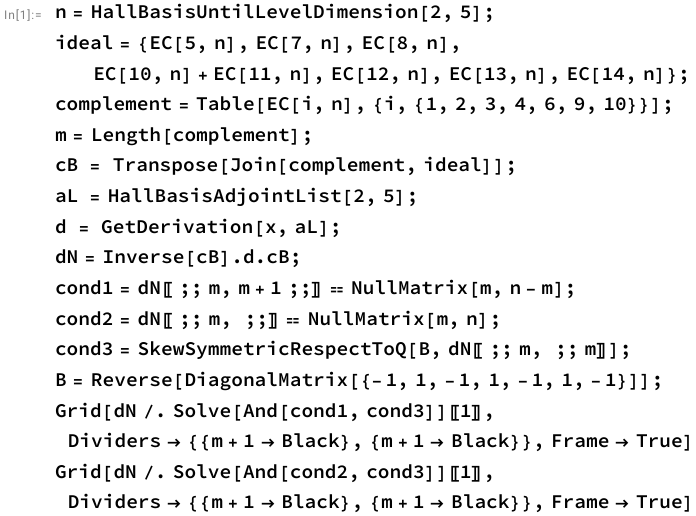}
		\caption{$\varphi$-skew derivations of $\mfn_{2,5}/I_1$ computed in Wolfram Mathematica using our library.}
		\label{fig:mat1}
	\end{figure}
\end{example}

Despite this approach is interesting from an algebraic point of view, it might not be the most efficient computationally for our purpose.

\subsection{Direct computation of derivations}

A more efficient approach involves directly computing derivations after obtaining the quotient algebra. We can accomplish this by using the same idea employed for finding derivations before the quotient. As we know how the structure constants (adjoint matrices) are after the quotient, we can directly compute its derivations without using any theorem. To aid in this computation, we have included another method to find the product in the quotient: \texttt{AdjointListQuotient[cB,iB,adjointList]}. Here, \texttt{cB} denotes a basis of $C$ with elements $\{c_1,\dots,c_n\}$, and \texttt{iB} is a basis of an ideal $I$, both expressed by their coordinates, such that $L = C \oplus I$. This method returns the list of adjoints of $L/I$ in the basis $\{c_1+I,\dots,c_n+I\}$.

\begin{example}\label{ex:tec2}
	In Figure~\ref{fig:mat2} we can see a Mathematica code in action to find $\varphi$-skew derivations of $\mfn_{2,5}/I_1$ finding first the product in the quotient algebra.
	\begin{figure}[h!]
		\centering
		\includegraphics{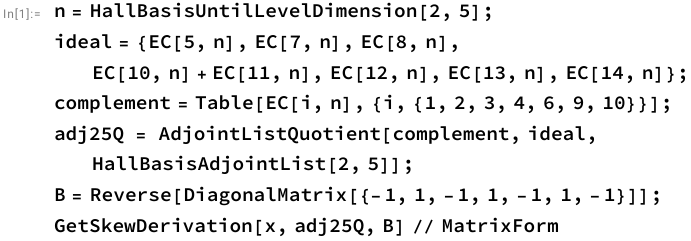}
		\caption{$\varphi$-skew derivations of $\mfn_{2,5}/I_1$ computed in Wolfram Mathematica using our library.}
		\label{fig:mat2}
	\end{figure}
\end{example}

\section{Computations and constructions}\label{s:computations}

Previous techniques serve us to compute all $\varphi$-skew derivations for the algebras in the classification given in Theorem~\ref{thm:clas_comp}. Next we are going to give this list, including the derivations calculated in Examples~\ref{ex:tec1} and \ref{ex:tec2}. Note each derivation is going to be expressed via its coordinate matrix in the defined basis using some variables $m_i$. Also, we are including some vertical and horizontal lines to help us distinguish the terms $\mfn^k$ of the lower central series and their quotiens $\mfn^k/\mfn^{k+1}$.

\begin{itemize}
	\item In $(\mfn_1, \varphi)$ all $\varphi$-skew derivations are null.
	\item In $(\mfn_2, \varphi)$ all $\varphi$-skew derivations are of the form
	      \begin{equation*}
		      \left(
		      \begin{array}{cc|c|cc}
			      m_1 & m_2  & 0    & 0   & 0    \\
			      m_3 & -m_1 & 0    & 0   & 0    \\
			      \hline
			      m_4 & m_5  & 0    & 0   & 0    \\
			      \hline
			      m_6 & 0    & m_5  & m_1 & m_2  \\
			      0   & m_6  & -m_4 & m_3 & -m_1 \\
		      \end{array}
		      \right)
	      \end{equation*}
	      Among them, inner derivations are the ones where $m_1 = m_2 = m_3 = 0$.
	\item In $(\mfn_3, \varphi)$ all $\varphi$-skew derivations are of the form
	      \begin{equation*}
		      \left(
		      \begin{array}{cc|c|c|c|cc}
			      m_1 & 0      & 0    & 0   & 0    & 0     & 0    \\
			      m_2 & -2 m_1 & 0    & 0   & 0    & 0     & 0    \\
			      \hline
			      m_3 & m_4    & -m_1 & 0   & 0    & 0     & 0    \\
			      \hline
			      m_5 & 0      & m_4  & 0   & 0    & 0     & 0    \\
			      \hline
			      m_6 & m_7    & 0    & m_4 & m_1  & 0     & 0    \\
			      \hline
			      m_8 & 0      & m_7  & 0   & m_4  & 2 m_1 & 0    \\
			      0   & m_8    & -m_6 & m_5 & -m_3 & m_2   & -m_1 \\
		      \end{array}
		      \right)
	      \end{equation*}
	      Among them, inner derivations are the ones where $m_1 = m_2 = m_7 = 0$.
	\item In $(\mfn_4, \varphi)$ all $\varphi$-skew derivations are of the form
	      \begin{equation*}
		      \left(
		      \begin{array}{cc|c|cc|c|cc}
			      0    & m_1 & 0    & 0    & 0    & 0    & 0    & 0   \\
			      -m_1 & 0   & 0    & 0    & 0    & 0    & 0    & 0   \\
			      \hline
			      m_2  & m_3 & 0    & 0    & 0    & 0    & 0    & 0   \\
			      \hline
			      m_4  & m_5 & m_3  & 0    & m_1  & 0    & 0    & 0   \\
			      m_5  & m_6 & -m_2 & -m_1 & 0    & 0    & 0    & 0   \\
			      \hline
			      m_7  & m_8 & 0    & m_3  & -m_2 & 0    & 0    & 0   \\
			      \hline
			      m_9  & 0   & m_8  & -m_5 & -m_6 & m_3  & 0    & m_1 \\
			      0    & m_9 & -m_7 & m_4  & m_5  & -m_2 & -m_1 & 0   \\
		      \end{array}
		      \right)
	      \end{equation*}
	      Among them, inner derivations are the ones where $m_1 = m_5 = 0$ and $m_4 = m_6$.
	\item In $(\mfn_5, \varphi)$ all $\varphi$-skew derivations are of the form
	      \begin{equation*}
		      \left(
		      \begin{array}{ccc|ccc}
			      m_1    & m_2    & m_3      & 0       & 0    & 0    \\
			      m_4    & m_5    & m_6      & 0       & 0    & 0    \\
			      m_7    & m_8    & -m_1-m_5 & 0       & 0    & 0    \\
			      \hline
			      m_9    & m_{10} & 0        & m_1+m_5 & m_6  & -m_3 \\
			      m_{11} & 0      & m_{10}   & m_8     & -m_5 & m_2  \\
			      0      & m_{11} & -m_9     & -m_7    & m_4  & -m_1 \\
		      \end{array}
		      \right)
	      \end{equation*}
	      Among them, inner derivations are the ones where $m_i = 0$ for $i \leq 8$.
	\item In $(\mfn_6, \varphi)$ all $\varphi$-skew derivations are of the form
	      \begin{equation*}
		      \left(
		      \begin{array}{ccc|cc|ccc}
			      m_1    & 0       & 0      & 0    & 0    & 0    & 0    & 0   \\
			      m_2    & -m_1    & m_3    & 0    & 0    & 0    & 0    & 0   \\
			      m_4    & -m_3    & -m_1   & 0    & 0    & 0    & 0    & 0   \\
			      \hline
			      m_5    & m_6     & m_7    & 0    & m_3  & 0    & 0    & 0   \\
			      m_8    & m_7     & m_9    & -m_3 & 0    & 0    & 0    & 0   \\
			      \hline
			      m_{10} & 0       & m_{11} & m_6  & m_7  & m_1  & 0    & m_3 \\
			      0      & m_{10}  & m_{12} & -m_5 & -m_8 & m_2  & -m_1 & m_4 \\
			      m_{12} & -m_{11} & 0      & m_7  & m_9  & -m_3 & 0    & m_1 \\
		      \end{array}
		      \right)
	      \end{equation*}
	      Among them, inner derivations are the ones where $m_i = 0$ for $i\leq 4$, also $m_7=m_{11} = 0$ and $m_6 = m_9$.
	\item In $(\mfn_7, \varphi)$ all $\varphi$-skew derivations are of the form
	      \begin{equation*}
		      \arraycolsep=2.3pt
		      \left(
		      \begin{array}{ccc|ccc|ccc}
			      0       & m_1    & m_2     & 0    & 0    & 0       & 0    & 0   & 0    \\
			      -m_1    & 0      & m_3     & 0    & 0    & 0       & 0    & 0   & 0    \\
			      -m_2    & -m_3   & 0       & 0    & 0    & 0       & 0    & 0   & 0    \\
			      \hline
			      m_4     & m_5    & m_6     & 0    & m_3  & -m_2    & 0    & 0   & 0    \\
			      m_7     & m_8    & m_9     & -m_3 & 0    & m_1     & 0    & 0   & 0    \\
			      m_8-m_6 & m_{10} & m_{11}  & m_2  & -m_1 & 0       & 0    & 0   & 0    \\
			      \hline
			      m_{12}  & 0      & -m_{14} & m_5  & m_8  & m_{10}  & 0    & m_1 & m_3  \\
			      0       & m_{12} & m_{13}  & -m_4 & -m_7 & m_6-m_8 & -m_1 & 0   & -m_2 \\
			      m_{13}  & m_{14} & 0       & m_6  & m_9  & m_{11}  & -m_3 & m_2 & 0    \\
		      \end{array}
		      \right)
	      \end{equation*}
	      Among them, inner derivations are the ones where $m_4 = m_{11}$, $m_5 = m_9$, $m_7 = m_{10}$ and $m_1 = m_2 = m_3 = m_6 = m_8 = 0$.
\end{itemize}

\begin{remark}
	Note the first block column of each derivation, which corresponds to the value of the derivation of the generators, defines uniquely the entire derivation.
\end{remark}

Once we have studied the derivations of each of our nilpotent Lie algebras, we are ready to find some simple subalgebras in their algebra of derivations. This would allow us, via double extensions, to obtain non-solvable and non-semisimple quadratic Lie algebras with no simple ideals\footnote{In case there is some simple ideal the algebra can be decomposed.}.
In light of the results, observing their upper left corners, just a few of them are valid for this purpose. All $\der \mfn_1$, $\der \mfn_3$, $\der \mfn_4$ and $\der \mfn_6$ solvable. Therefore, only $\mfn_2$, $\mfn_5$ and $\mfn_7$ can be double extended to produce non-solvable quadratic Lie algebras. The Lie algebras we can obtain in this cases are:
\begin{itemize}
	\item For $\mfn_2$, which is isomorphic to $\mfn_{2,3}$, there is a subalgebra of derivations isomorphic to $\mfsl_2(\mbF)$, just considering $\varphi$-skew derivations where $m_i = 0$ for $i\geq 4$. This subalgebra is just a Levi factor of $\der_\varphi \mfn_{2,3}$ which is isomorphic to $\mfsl_2(\mbF)$. Using the double extensions method we obtain the algebra $\mfsl_2(\mbF) \oplus \mfn_{2,3} \oplus \mfsl_2(\mbF)^*$.
	\item For $\mfn_5$, which is isomorphic to $\mfn_{3,2}$, there is a subalgebra of derivations isomorphic to $\mfsl_3(\mbF)$, just considering $\varphi$-skew derivations where $m_i = 0$ for $i\geq 9$. This subalgebra is just a Levi factor of $\der_\varphi \mfn_{3,2}$ an it is isomorphic to $\mfsl_3(\mbF)$. Using the double extension procedure we get the algebra $\mfsl_3(\mbF) \oplus \mfn_{3,2} \oplus \mfsl_3(\mbF)^*$. Also, we can easily produce $\mfsl_2(\mbF) \oplus \mfn_{3,2} \oplus \mfsl_2(\mbF)^*$.
	\item For $\mfn_7$, which is isomorphic to $\mfn_{3,3}/I$ with
	      \begin{multline*}
		      I = \spa\langle
		      [[x_2,x_3],x_1],
		      [[x_1,x_3],x_2],
		      [[x_2,x_3],x_2] - [[x_1,x_3],x_1], \\
		      [[x_1,x_3],x_3] - [[x_1,x_2],x_2],
		      [[x_2,x_3],x_3] + [[x_1,x_2],x_1]
		      \rangle,
	      \end{multline*}
	      there is a subalgebra of derivations isomorphic to $\mfsl_2(\mbF)$, obtained by considering $\varphi$-skew derivations where $m_i = 0$ for $i\geq 4$. This subalgebra is just a Levi factor of $\der_\varphi \mfn_7$ an it is isomorphic to $\mfso_3(\mbF)\cong \mfsl_2(\mbF)$. Using the double extension construction, we obtain algebras isomorphic to $\mfsl_2(\mbF) \oplus (\mfn_{3,3}/I) \oplus \mfsl_2(\mbF)^*$.
\end{itemize}
All these Levi factors can also be obtained via algorithms. However, as in this case, they can be observed directly, we have not implemented them. Nevertheless, in~\cite{DeGraaf_2000} we can find a wide range of algorithms in pseudocode, including not only the Levi factor one, but also some algorithms mentioned in this paper (derivations, quotient algebras, Hall basis...) or even other ones that can be useful to compute the nilradical or radical of some Lie algebras.

\section{Conclusions}

This paper presents a method to obtain non-solvable and non-semisimple quadratic Lie algebras whose largest solvable radical is nilpotent. The process begins with the deconstruction of these non-solvable and non-semisimple quadratic Lie algebras. All these algebras can be obtained via double extensions from their solvable radical, provided they do not have simple ideals that would lead to reducible quadratic Lie algebras. To study our specific case, we use as support a classification of some small nilpotent Lie algebras using their isomorphisms to quotients of free nilpotent Lie algebras. After that, we are ready to find their $\varphi$-skew derivations to determine all their possible double extensions. These derivations can be obtained in two different ways: finding how derivations in a quotient work, or finding the new Lie product and computing the derivations directly.

Although the techniques used in this article may seem particular, most of them can be extended. For instance, all techniques can be applied for larger nilpotent quadratic Lie algebras. And even more, most can be used for general solvable quadratic Lie algebras. Moreover, throughout this study, we have found and completed some general results and properties of these algebras. For example, we have precisely identify the relationship between non-solvable quadratic Lie algebras with no simple ideals and the double extension producing them from their radical.

Finally, but not less important, most of the procedures described in this paper have been implemented computationally. In fact, we have released a Wolfram Mathematica Library with all these functions and more, ready to be utilized. This tool is highly beneficial for obtaining examples, understanding how it works, and getting conjectures of what is happening in order to develop new results after. This tool can be complemented with other Lie algebras packages as some found in GAP4, LiE or Magma. Although some algorithms have been defined before, we have implemented them over a different language and we have developed new methods such as \verb|GetNilpotentIdeal| or \verb|SkewDerivations| which are a key part of our work. In addition, all those methods are compatible among themselves and they can be nested to study complex situations.

\section*{Acknowledgements}

This research has been partially funded by grant Fortalece 2023/03 of ``Comunidad Aut\'onoma de La Rioja'' and by grant MTM2017-83506-C2-1-P of ``Ministerio de Econom\'ia, Industria y Competitividad, Gobierno de Espa\~na'' (Spain) from 2017 to 2022 and by grant PID2021-123461NB-C21, funded by MCIN/AEI/10.13039/501100011033 and by ``ERDF: A way of making Europe'' starting from 2023. The third author has been also supported until 2022 by a predoctoral research contract FPI-2018 of ``Universidad de La Rioja''.

\bibliographystyle{apalike}
\bibliography{bibliography.bib}

\end{document}